\documentclass[11pt,letterpaper]{amsart}

\usepackage{mathrsfs}
\usepackage{cite}
\usepackage{color}
\usepackage{amssymb}
\theoremstyle{plain}

\author[K. Zhang]{Kaiqiang Zhang
}
\address{Kaiqiang Zhang, School of Computer Science and Technology, Dongguan University of Technology, 523808 Dongguan, China}
\email{math.zhangkq@dgut.edu.cn}

\subjclass[2020]{35B44, 35C06, 35B35, 35K57, 35Q92}

\begin{document}

\numberwithin{equation}{section}
\renewcommand{\theequation}{\arabic{section}.\arabic{equation}}
\theoremstyle{plain}
\newtheorem{exam}{Example}[section]
\newtheorem{theorem}[exam]{Theorem}
\newtheorem{lemma}[exam]{Lemma}
\newtheorem{remark}[exam]{Remark}
\newtheorem{hyp}[exam]{Hypothesis}
\newtheorem{proposition}[exam]{Proposition}
\newtheorem{definition}[exam]{Definition}
\newtheorem{corollary}[exam]{Corollary}
\newtheorem{analytic}[exam]{Analytic extension principle}
\newtheorem{notation}[exam]{Notation}
\setlength{\baselineskip}{1.5\baselineskip}

\title[Inhomogeneous parabolic equation]{{Finite time blow-up for an inhomogeneous parabolic equation}}
\begin{abstract}
We consider the inhomogeneous nonlinear heat equation
\[
\partial_t u-\Delta u=|u|^{p-1}u+f(x),\qquad x\in\mathbb{R}^3,\quad p>5,
\]
where \(f\in L^\infty\cap C^{0,1}(\mathbb{R}^3)\). For every sufficiently large integer \(n\), we construct a codimension-\(n\) Lipschitz manifold of non-radial initial data whose corresponding solutions blow up in finite time and whose rescaled profiles converge to the prescribed self-similar profile \(\Phi_n\) of the homogeneous equation.

The main novelty is to show that the finite-codimensional stability mechanism for self-similar blow-up, developed in the work of Collot, Raphaël and Szeftel [Mem. Amer. Math. Soc. (2019)] for the homogeneous equation, is robust under the addition of a bounded, Lipschitz spatially inhomogeneous source term. In contrast with the homogeneous problem, the equation considered here has no exact scaling invariance, which is a key ingredient in many previous constructions. We expect that the  framework developed in this work may also be useful for related problems in which exact scaling invariance is broken.
\end{abstract}

\maketitle

\section{Introduction}

We consider the inhomogeneous parabolic equation
\begin{equation}\label{1.1}
\left\{\begin{array}{l}
\partial_{t} u-\Delta u=|u|^{p-1}u+f(x),  \\
u(0,x)=u_0(x),
\end{array}\quad \text{in}\ \mathbb{R}^{d},\ p>1,\right.
\end{equation}
where $f(x)\in L^\infty\cap C^{0,1}(\mathbb{R}^d)$.  
We refer the reader to the recent works 
\cite{Ishige-Kawakami-Takada-2026,Hisa-Jin-2020} for the local well-posedness results. 
Finite-time blow-up solutions for problem \eqref{1.1} have been studied in 
\cite{Bandle-Levine-Zhang-2000,Zhang-1999,Zhang-1998,Wang-Yin-You-2023,Luo-Yin-You-2024,Luo-Yin-You-2025}, but most of the works focus on blow-up criteria and lifespan estimates. 
They leave open the problem of constructing smooth localized initial data whose corresponding solutions blow up with a precise asymptotic profile, as well as describing the associated stability structure.
This question has been answered affirmatively for the Fujita equation 
\cite{Collot-2020,Collot-Pierre-2019}, the Keller--Segel equation 
\cite{Colllot-Zhang,Li-Zhou-2025},  and Hénon-type parabolic equations 
\cite{Irfan-Sarah-Birgit-2026}. 
For problem \eqref{1.1}, to the best of our knowledge, no such precise description of the blow-up dynamics has been obtained so far.
The purpose of this paper is to address this question.

Let \(T\) denote the maximal existence time of a solution \(u\). 
If \(T<\infty\), we say that blow-up occurs in the sense that
\[
\limsup_{t\to T^-}\|u(t)\|_{L^\infty(\mathbb{R}^d)}=+\infty.
\]
A point \(x_0\in\mathbb R^d\) is called a blow-up point if, for every
\(\delta>0\),
\[
\limsup_{t\to T^-}
\|u(t)\|_{L^\infty(B(x_0,\delta))}=+\infty.
\]

\subsection{Background}
\subsubsection{Finite time blow-up solutions of the classical nonlinear heat equation}
When \(f(x)\equiv0\), equation \eqref{1.1} reduces to the classical nonlinear heat
equation
\begin{equation}\label{0.1}
\partial_t u-\Delta u=|u|^{p-1}u,\ \ \ x\in \mathbb{R}^d.
\end{equation}
The study of finite-time blow-up solutions to \eqref{0.1} goes back to the seminal work
of Fujita \cite{Fujita} and has since developed into a central topic in nonlinear
parabolic equations.  Two important directions have emerged: one concerns criteria for global existence versus finite-time blow-up \cite{Lee-1992,Fujita,Zhang-Li-2026,Ikehata-2010,Harada-2026}, while the other focuses on the precise dynamics of singularity formation, including blow-up rates, profiles, and stability properties \cite{Collot-2020,Pino-Musso-Wei-2021,Collot-2018,Pino-Musso-Wei-zhou-2020}.
We refer to \cite{Quittner-Souplet-2019} for a comprehensive
account.

The characterization of blow-up solutions for equation \eqref{0.1} has been the subject of an extensive body of literature, and we recall below the key result directly relevant to our analysis.
The equation \eqref{0.1} is invariant under the following scaling: for any $\lambda>0$, if $u(t,x)$ solves \eqref{0.1}, then so is
$$
u_\lambda(t,x)=\lambda^{\frac{2}{p-1}}u(\lambda^2t,\lambda x),
$$ with the rescaled initial data $\lambda^{\frac{2}{p-1}}u_0(\lambda x)$.
Discarding the diffusion term in \eqref{0.1} yields the equation $u_t=u^p$, whose spatially homogeneous solution takes the form $(p-1)^{-\frac{1}{p-1}}(T-t)^{-\frac{1}{p-1}}$. It is said that the
blow-up is of type I if a solution $u$ of \eqref{0.1}
blows up at the ODE rate, i.e.,
\[
\limsup_{t\to T} (T-t)^{\frac{1}{p-1}}\lVert u(t) \rVert_{L^\infty} < \infty,
\]otherwise, the blow-up is of type II. Self-similar blow-up  solutions provide important examples of type I blow-up dynamics. They have the
form:
$$
u(t,x)=\frac{1}{(T-t)^{\frac{1}{p-1}}}v(y),\ \ \ y=\frac{x}{\sqrt{T-t}},
$$
where the profile \(v\) solves
\begin{equation}\label{self-similar-equations}
  \Delta v-\frac{1}{2}\Lambda v+v^p=0,\ \ \Lambda v=\frac{2}{p-1}v+y\cdot\nabla v.
\end{equation}
The solutions of \eqref{self-similar-equations} have been constructed by the ODE approach in \cite{Budd-Norbury-1987,Budd-Qi-1989,Troy-1987}.
For \(d=3\) and \(p>5\), using the normalization $y={x}/{\sqrt{2(T-t)}}$, under which the self-similar profile equation becomes $\Delta v-\Lambda v+v^p=0$.
 Collot, Rapha\"el and Szeftel
\cite{Collot-Pierre-2019} constructed a countable family of smooth self-similar profiles 
\(\Phi_n\) by combining matched asymptotic expansions with the Banach fixed point theorem. They also established a finite-codimensional non-radial stability analysis around these profiles. These profiles $\Phi_n$ constitute the building block of our paper.
For self-similar blow-up solutions of other evolution equations, we refer the reader to \cite{Liu-Raees-2025,Roland-Ostermann-2023} for the wave equation, \cite{Roland-Schörkhuber-2026} for the Schrödinger equation, and \cite{Colllot-Zhang,Li-Zhou-2025,Nguyen-Wang-Zhang-2026} for the Keller–Segel system.
\subsubsection{Finite time blow-up solution of the inhomogeneous nonlinear heat equation}
It is natural to ask whether the addition of the spatially inhomogeneous term \(f(x)\) can induce a blow-up mechanism different from that of the homogeneous equation.  A striking difference is that the inhomogeneous problem \eqref{1.1} no longer 
possesses the scaling invariance of the homogeneous equation \eqref{0.1}.  This lack of invariance suggests that the spatial profile of \(f(x)\) may interact with the nonlinear dynamics and  affect both the location and the structure of blow-up.

The inhomogeneous problem \eqref{1.1} is less understood from the viewpoint of
precise singularity formation. Previous works have mainly focused on blow-up
criteria, global existence versus finite-time blow-up, and lifespan estimates.
For instance, Bandle, Levine and Q. Zhang \cite{Bandle-Levine-Zhang-2000} showed that for $d\ge3$, $\int_{\mathbb{R}^d}f(x)dx>0$ and $1<p<d/(d-2)$, all the solutions of \eqref{1.1} blow up  in finite time,
while if  $p>d/(d-2)$ there exist both global and non-global solutions. Subsequently, the critical exponent $p=n/(n-2)$ was shown to belong to the blow-up regime in \cite{Kartsatos-2004,Zeng-2007}.  Related results on complete Riemannian manifolds
were obtained in \cite{Zhang-1999,Zhang-1998}. 
More recently, the influence
of the source term \(f\) on the lifespan of solutions has been studied in
\cite{Wang-Yin-You-2023,Luo-Yin-You-2024,Luo-Yin-You-2025}. For solutions with zero initial data, the life span of the solution to the problem \eqref{1.1} was analyzed in \cite{Wang-Yin-You-2023,Luo-Yin-You-2024}, and a sharp life span estimate was later established for non-zero initial data \cite{Luo-Yin-You-2025}.

Despite the literature on the blow-up criteria and blow-up time, to our knowledge, these existing results do not describe the detailed asymptotic structure of
blow-up solutions.  Precisely, no previous result appears to describe the blow-up rate, the blow-up profile, and the blow-up point of such solutions. In this paper, we employ modulation techniques and a dynamical rescaling formulation 
to construct finite-time blow-up solutions whose rescaled profiles converge to a prescribed \(\Phi_n\) of the self-similar profile family of the homogeneous equation.
This construction strategy is motivated by the singularity formation results of the nonlinear
heat equation \cite{Collot-Pierre-2019,LiSun-2026} and Keller-Segel equation \cite{Colllot-Zhang}.  For other techniques involving the construction of blow-up solutions for the parabolic equations, we refer the reader to \cite{Hou-Nguyen-Wang-2026} for the $L^2$-based methodology and \cite{Dávila-Pino-Wei-2020,Cortázar-Pino-Musso-2020} for the inner-outer gluing method.

\subsection{Statement of the main result}
This paper aims to fill a gap in the literature concerning the precise
asymptotic behavior of finite-time blow-up solutions of the inhomogeneous
parabolic equation \eqref{1.1}. 

Most previous constructions of self-similar blow-up solutions rely crucially on the scaling invariance of the underlying equation \cite{Collot-Pierre-2019,Colllot-Zhang,Glogi-2025}. This invariance leads, after passing to self-similar variables, to an autonomous renormalized flow around a stationary self-similar profile. In the present inhomogeneous problem, however, the forcing term \(f(x)\)
breaks the scaling invariance. After passing to the backward self-similar
variables
\[
s=-\log(T-t),\qquad y=\frac{x-x(t)}{\sqrt{T-t}},
\]
the forcing term becomes
\[
e^{-\frac{ps}{p-1}} f\bigl(x(t)+e^{-s/2}y\bigr),
\]
which depends explicitly on the similarity time \(s\) and the modulation
parameter \(x(t)\), and hence gives rise to a non-autonomous perturbation of
the renormalized flow.

This observation leads to the main difficulty of the paper. Unlike in the homogeneous case, the renormalized dynamics is no longer autonomous. The main point of the analysis is to show that this additional non-autonomous perturbation remains lower order along the blow-up dynamics and that the finite-codimensional stability mechanism constructed in \cite{Collot-Pierre-2019} persists in the present setting. This persistence does not require any smallness assumption on \(\|f\|_{L^\infty(\mathbb R^3)}\). This is achieved by choosing the initial modulation scale \(\lambda_0\) sufficiently small, equivalently the initial renormalized time \(s_0\) sufficiently large, so that the rescaled forcing term is a lower-order perturbation throughout the bootstrap regime. We expect that the strategy developed here may be useful for other parabolic equations in which exact scaling invariance is absent; see, for example, the equation studied in \cite{Zhangk1,SZ}.

We use the self-similar profiles $\Phi_n$ constructed in
\cite{Collot-Pierre-2019} to construct the finite-time blow-up solutions of \eqref{1.1}, see  Proposition \ref{self-similar profiles} in the next section for the asymptotic behavior of $\Phi_n$.
The following result establishes the existence and stability of finite-time blow-up solutions  to \eqref{1.1} for 
 $d=3$ and $p>5$. 
\begin{theorem}\label{thm}
Let $d=3$, $p>5$ and $f(x)\in L^\infty\cap C^{0,1}(\mathbb{R}^3)$. Then, for every sufficiently large integer \(n\),
for every fixed \(x_0\in\mathbb R^3\), and for every sufficiently small
\(\lambda_0>0\), there exists a codimension-\(n\) Lipschitz manifold
$
\mathcal M_{n,\lambda_0,x_0}\subset L^\infty(\mathbb R^3)$
containing non-radial initial data such that, for every
\(u_0\in\mathcal M_{n,\lambda_0,x_0}\), the corresponding solution to
\eqref{1.1} blows up at a finite time \(T\) satisfying
$
\frac18\lambda_0^2\le T\le 2\lambda_0^2$. Moreover, there exist a modulation parameter $x(t)\in\mathbb R^3$ and a remainder term $\bar u$ such that 
 $$
 u(t,x)=\frac{1}{[2(T-t)]^{\frac{1}{p-1}}}\left[\Phi_n\left(y\right)+\bar{u}\left(t,y\right)\right],\ \ y=\frac{x-x(t)}{\sqrt{2(T-t)}}.
 $$
The remainder satisfies  
 \begin{equation}\label{stability}
 \lim_{t\to T}||\bar{u}(t)||_{L^\infty(\mathbb{R}^3)}=0,
 \end{equation}
and the modulation parameter satisfies
 \begin{equation}\label{blow-up point}
x(t)\to x(T),   \ \ t\to T^-,
 \end{equation}
 where $x(T)$ is a blow-up point of the solution. In addition, the blow-up time map  $T:\mathcal M_{n,\lambda_0,x_0}\to \mathbb R$ is Lipschitz continuous with respect to the \(L^\infty\) topology. More precisely, for any \(u_0,u_0'\in\mathcal M_{n,\lambda_0,x_0}\), one has $|T(u_0)-T(u_0')| \lesssim  \|u_0-u_0'\|_{L^\infty(\mathbb R^3)}$. 
\end{theorem}

\textit{Comments on the result.}

1.  \textit{The choice of the parameter $n$}. The restriction \(n\gg1\) comes from the construction of the self-similar
profiles \(\Phi_n\) in \cite{Collot-Pierre-2019}. The index \(n\) corresponds to
the number of zeros of
$
\frac{2}{p-1}\Phi_n+x\cdot\nabla\Phi_n.
$
The existence and spectral properties of these profiles are currently available
only for sufficiently large \(n\).

2. \textit{Construction of the manifold}. 
We set $\phi_{j,n}=\psi_{j+1,n}$ ($j\ge1$) where $\psi_{j,n}$ are the eigenfunctions of the linear operator $L_n$ which is defined in \eqref{linear}; see Proposition
\ref{prop:spectral-gap}. We define
$$
V_n^\perp=\{\bar{u}\in L^\infty(\mathbb{R}^3);(\bar{u},\phi_{j,n})_\rho=0,\ 1\le j\le n\},
$$
where $(\cdot,\cdot)_\rho$ denotes the weighted inner product defined in \eqref{product}.
There exist $n$ Lipschitz functions $$a_j:V_n^\perp\to \mathbb{R},\ \ 1\le j\le n.$$
The manifold $\mathcal M_{n,\lambda_0,x_0}$ is then obtained as the graph of these functions
over the stable subspace $V_n^\perp$. For fixed $0<\lambda_0\ll1$ and $x_0\in\mathbb R^3$, define
$$
\mathcal{M}_{n,\lambda_0,x_0}=\left\{\frac{1}{\lambda_0^{\frac{2}{p-1}}}\left(\Phi_n+\bar{u}_0+\sum_{j=1}^na_j(\bar{u}_0)\phi_{j,n}\right)\left(\frac{x-x_0}{\lambda_0}\right);\bar{u}_0\in V_n^\perp,||\bar{u}_0||_{L^\infty}\ll1\right\}.
$$
This construction shows, in particular, that $\mathcal M_{n,\lambda_0,x_0}$ contains
non-radial initial data.
In the homogeneous case, the scaling invariance allows
one to choose any $\lambda_0>0$ to construct the initial data manifold $\mathcal M_{n,\lambda_0,x_0}$, and there holds $a_j(0)=0$ for $1\le j\le n$.
\\
3. \textit{Regularity of the manifold $\mathcal{M}_{n,\lambda_0,x_0}$}.
 It is natural to expect that these functions $a_j$ enjoy higher regularity, for
instance that $a_j$ are $C^k$ for every $k\ge 1$, or at least $C^1$, as in
the construction of the threshold manifold in \cite{Martel-Merle-Nakanishi-2016}. However,
such a higher regularity result is not obtained by the present argument and
remains an open problem.
\\
4. \textit{The existence of smooth and
localized initial data}. Indeed, one may simply take
$
\bar u_0=0\in V_n^\perp .$
Then the corresponding initial data is
\[
u_0(x)
=
\lambda_0^{-\frac{2}{p-1}}
\left(
\Phi_n+\sum_{j=1}^n a_j(0)\phi_{j,n}
\right)
\left(\frac{x-x_0}{\lambda_0}\right).
\]
By the smoothness and asymptotic behavior of the self-similar profile
\(\Phi_n\) and of the eigenfunctions \(\phi_{j,n}\), this initial data is smooth and has algebraic decay at infinity.
\subsection{Organization of the paper}
This paper is organized as follows. In Section \ref{section0}, we collect the preliminary results needed for the analysis, including the construction of the self-similar profiles and the spectral properties of the linearized operator around them. In Section \ref{section4}, we introduce the renormalized variables and formulate the corresponding dynamical system. We then close the bootstrap argument by means of the spectral gap estimate, derive the required modulation, energy, and $L^\infty$ bounds, and establish the Lipschitz dependence of both the unstable parameters and the blow-up time. These estimates complete the construction of the codimension-$n$ Lipschitz manifold and the proof of Theorem \ref{thm}.

\section{Preliminaries}\label{section0}
In this section, we give the preliminary results in \cite{Collot-Pierre-2019} and some notations which will be used in our later analysis.
We first recall the following existence result of the self-similar profiles of the homogeneous equation \eqref{0.1}, proved in \cite{Collot-Pierre-2019}.
\begin{proposition}[\cite{Collot-Pierre-2019}, Proposition 1.1]\label{self-similar profiles}
There exists $N$ such that for every integer $n>N$, there exists a smooth radial solution $\Phi_n$ to the  equation 
\begin{equation}\label{self-similar-equation}
   -\Delta \Phi+\Lambda\Phi-\Phi^p=0,\ \ \Lambda \Phi=\frac{2}{p-1}\Phi+y\cdot\nabla \Phi.
\end{equation}
Moreover, the radial function $\Lambda\Phi_n(r)$ has exactly $n$ zeros in
$(0,\infty)$.
There exists a sufficiently small constant $r_0>0$, independent of $n$, such
that the following asymptotic properties hold.
\begin{enumerate}
    \item Behavior away from the origin:
    \begin{equation}\nonumber
        \lim_{n \to +\infty} \sup_{r \geq r_0} \left( 1 + r^{\frac{2}{p-1}} \right) \left| \Phi_n(r) -\Phi_*(r) \right| = 0.
    \end{equation}
where $\Phi_*(r) =
    \left[
    \frac{2}{p-1}
    \left(d-2-\frac{2}{p-1}\right)
    \right]^{\frac{1}{p-1}}
    r^{-\frac{2}{p-1}}.$

    \item Behavior near the origin: there exists a sequence $\mu_n > 0$ with $\mu_n \to 0$ as $n \to +\infty$ such that
    \begin{equation}\nonumber
        \lim_{n \to +\infty} \sup_{r \leq r_0} \left| \Phi_n(r) - \frac{1}{\mu_n^{\frac{2}{p-1}}} Q \left( \frac{r}{\mu_n} \right) \right| = 0.
    \end{equation}
where $Q(r)$ is the unique radial solution to 
\begin{equation}\nonumber
\left\{
\begin{aligned}
&Q''+\frac{2}{r}Q'+Q^p=0,\\
&Q(0)=1,\ Q'(0)=0.
\end{aligned}
\right.
\end{equation}
Moreover, $
    Q(r)\sim \Phi_*(r)$ as $r\to+\infty.$
\end{enumerate}
\end{proposition}
By Proposition \ref{self-similar profiles}, the profiles \(\Phi_n\) are smooth
radially symmetric functions satisfying the asymptotic decay
\begin{equation}\label{large-r}
 \Phi_n(r)\sim r^{-\frac{2}{p-1}}
\quad \text{as } r\to\infty.   
\end{equation}

We define the linearized operator corresponding to \eqref{self-similar-equation} around $\Phi_n$ by
\begin{equation}\label{linear}
L_n=-\Delta+\Lambda-p\Phi_n^{p-1},
\end{equation}
where the operator $\Lambda$ defined in \eqref{self-similar-equation}.
We define the scalar product
\begin{equation}\label{product}
(f,g)_\rho=\int_{\mathbb{R}^3}f(x)g(x)\rho(x)dx,\ \  \rho(x)=e^{-\frac{|x|^2}{2}},    
\end{equation}
and let $L^2_\rho$ be the weighted Sobolev space. We define $H^k_{\rho}(\mathbb{R}^3)$ by
$$
||u||_{H^k_\rho}=\sqrt{\sum_{i=0}^k||D^iu||^2_{L^2_\rho}}.
$$
By direct computation, we know that $L_n$ is a self-adjoint operator on the space $L^2_\rho$. In the following proposition, we recall some spectral properties of the operator $L_n$ which were presented in \cite{Collot-Pierre-2019}.

\begin{proposition}
[\cite{Collot-Pierre-2019}, Proposition 3.1]\label{prop:spectral-gap}
Let $ n > N\gg 1 $, then the following holds:

\begin{enumerate}
    \item \text{Eigenvalues and eigenfunctions.} The spectrum of $L_n$ is given by
    \begin{equation}\nonumber
        -\mu_{n+1,n} < \cdots < -\mu_{2,n} < -\mu_{1,n} = -2 < -\mu_{-1,n} = -1 < 0 < \lambda_{0,n} < \lambda_{1,n} < \cdots
    \end{equation}
    with
    \begin{equation}\nonumber
        \lambda_{j,n} > 0 \text{ for all } j \geq 0 \text{ and } \lim_{j \to +\infty} \lambda_{j,n} = +\infty.
    \end{equation}
The eigenvalues $(-\mu_{j,n})_{1 \leq j \leq n+1}$ are simple and associated to spherically symmetric bounded eigenvectors $\psi_{j,n}$ with
 \[
 \| \psi_{j,n} \|_{L^2_\rho} = 1, \quad \psi_{1,n} = \frac{\Lambda \Phi_n}{\|\Lambda \Phi_n\|_\rho},
 \]
and the eigenspace for $\mu_{-1,n}$ is spanned by
 \begin{equation}\nonumber
\psi_{-1,n}^k = \frac{\partial_k \Phi_n}{\|\partial_k \Phi_n\|_\rho}, \quad 1 \leq k \leq 3.
\end{equation}
 Moreover, there holds as $ r \to +\infty $
    \begin{equation}\nonumber
        |\partial_k \psi_{j,n}(r)| \lesssim (1 + r)^{-\frac{2}{p-1}-\mu_{j,n}-k}, \quad 1 \leq j \leq n+1, \quad k \geq 0.
    \end{equation}
\item \text{Spectral gap.} There holds for some constant $ c_n > 0 $:
    \begin{equation}\label{spectral-gap}
        \forall \varepsilon \in H^1_\rho, \, (L_n \varepsilon, \varepsilon)_\rho \geq c_n \| \varepsilon \|^2_{H^1_\rho} - \frac{1}{c_n} \left[ \sum_{j=1}^{n+1} (\varepsilon, \psi_{j,n})_\rho^2 + \sum_{k=1}^3 (\varepsilon, \psi_{-1,n}^k)^2_\rho \right].
    \end{equation}
\end{enumerate}
\end{proposition}
\begin{remark}
From the above proposition, we know that the operator $L_n$ admits \(n+4\) unstable directions: \(n+1\)
spherically symmetric modes \(\psi_{j,n}\) ($1\le j\le n+1$), and three
translation modes \(\partial_k\Phi_n\) $(1\le k\le 3)$.   

In the homogeneous case ($f\equiv 0$), the mode
$
\psi_{1,n}=\frac{\Lambda\Phi_n}{||\Lambda\Phi_n||_{\rho}}
$
arises from varying the blow-up scale, or equivalently the blow-up time, while the modes
$
\partial_k\Phi_n$ ($ 1\le k\le 3),$
arise from varying the blow-up point. In the present inhomogeneous problem \eqref{1.1}, these directions are no longer generated by exact symmetries, since the source term $f(x)$ breaks both the scaling and translation invariances. Nevertheless, they are still controlled through the modulation parameters $\lambda(s)$ and $x(s)$ in the next section. Hence they are not eliminated by imposing codimension conditions on the initial data. The only unstable directions that have to be removed in this way are the remaining $n$ modes
$\psi_{2,n},\ldots,\psi_{n+1,n}$,
which explains why the set constructed in Theorem~\ref{thm} has codimension $n$.
\end{remark}

\textbf{Notations}. In the rest of this paper, we restrict ourselves to the case $d=3$ and $p>5$. Let $B_R$ denote the ball of radius $R>0$ centered at the origin in $\mathbb{R}^3$.
We use the notation $a\lesssim b$ if there exists an independent constant $C$ such that $a\le Cb$. The notation $C(\cdot)$ denotes a positive
constant depending only on “·”, and its value may change from line to line.

\section{Dynamical control of the flow}\label{section4}

In the rest of this paper, $n$ is fixed ($n\gg1$). For simplicity, we omit the $n$ subscript and write $\psi_j$, $\mu_j$ and $\lambda_j$ instead.
\subsection{Renormalisation}

We define the $L^\infty$ tube around the renormalized versions of $\Phi_n$:
$$
X_\delta=\bigg\{u=\frac{1}{\mu^{\frac{2}{p-1}}}(\Phi_n+\bar{u})\bigg(\frac{x-x'}{\mu}\bigg),\  \mu>0,\ ||\bar{u}||_{L^\infty}<\delta\bigg\}.
$$
The following geometrical decomposition lemma is essentially contained in 
\cite{Collot-Pierre-2019}.  We provide a slightly more detailed proof adapted to the present setting.
\begin{lemma}[Geometrical decomposition]\label{Implicit}
 There exist  $\delta>0$  and $ C>0$  such that any  $u \in X_{\delta}$ admits a unique decomposition
$$u=\frac{1}{\lambda^{\frac{2}{p-1}}}\left(\Phi_n+\sum_{j=2}^{n+1}a_j\psi_j+\varepsilon\right)\bigg(\frac{x-\bar{x}}{\lambda}\bigg),$$
where  $\varepsilon$  satisfies the orthogonality conditions
$$ \left(\varepsilon,\psi_j\right)_{\rho}=(\varepsilon,\partial_k\Phi_n)_\rho=0,\ \ 1\le j\le n+1,\ \ 1\le k\le 3.$$
Moreover, 
the maps
\[
(\bar u,\mu,x')\mapsto \lambda,\qquad
(\bar u,\mu,x')\mapsto \bar x,\qquad
(\bar u,\mu,x')\mapsto a_j,\quad 2\le j\le n+1,
\]
are smooth. In addition, the following estimate holds: 
$$
||\varepsilon||_{L^\infty}+\left|
\frac{\lambda}{\mu}-1\right|+\sum_{j=2}^{n+1}|a_j|+\frac{|\bar{x}-x'|}{\mu}\le C\delta.$$
\end{lemma}
\begin{proof}
We define the map:
\begin{equation}\label{map}
   F(\bar{u}, \mu, x, b_2,\cdots, b_{n+1})=\mu^{\frac{2}{p-1}}(\Phi_n+\bar{u})(\mu y+x)-\Phi_n(y)-\sum_{j=2}^{n+1}b_j\psi_j(y). 
\end{equation}
Then we define the vector-valued function$$G(\bar{u}, \mu, x, b_2,\cdots, b_{n+1}):=((F,\Lambda \Phi_n)_\rho,(F,\partial_1 \Phi_n)_\rho,(F,\partial_2 \Phi_n)_\rho,(F,\partial_3 \Phi_n)_\rho,(F,\psi_2)_\rho,\cdots,(F,\psi_{n+1})_\rho).$$
We immediately check that  $G(0,1,0\cdots,0)=(0,0\cdots,0)$. Since the operator $L_n$  is  self-adjoint for the $L^2_{\rho}$ product, there holds \begin{equation}\label{nonbus}
(\psi_j,\psi_k)_\rho=\begin{cases}1,& \text { for } j=k, \\ 0,& \text { for } j\neq k.\end{cases}
\end{equation} We thus know that the  Jacobian matrix
$$
\left.\frac{\partial G}{\partial (\mu,x,b_2,\cdots, b_{n+1})}\right|_{(0,1,0,\cdots,0)}
=\left(\begin{array}{cccc}
||\Lambda\Phi_n||^2_{L^2_\rho} & 0 & \cdots & 0 \vspace{1ex}\\
0 & ||\partial_1\Phi_n||^2_{L^2_\rho} & \cdots & 0 \\
\cdots & & & \\
0 &0 & \cdots & -1
\end{array}\right)
$$
is invertible. Although the map $F$ is not regarded as a smooth $L^\infty$-valued map, the map $G$ is smooth as a finite-dimensional map. 
By the implicit function theorem, there exist $\delta>0$ and unique smooth
maps
\begin{equation}\label{Smooth-map}
\bar u\mapsto
\bar\lambda(\bar u),\qquad
\bar u\mapsto
\widetilde x(\bar u),\qquad
\bar u\mapsto
a_j(\bar u),\quad 2\le j\le n+1,
\end{equation}
defined for $\|\bar u\|_{L^\infty}<\delta$, such that
\[
G(\bar u,\bar\lambda(\bar u),\widetilde x(\bar u),
a_2(\bar u),\ldots,a_{n+1}(\bar u))=0.
\]
Equivalently,
\begin{equation}\label{varepsilon}
\varepsilon(y)=F(\bar{u}, \bar{\lambda},\tilde{x}, a_2,\cdots,a_{n+1})\end{equation} satisfying
\begin{equation}\label{varepsilon1}\left(\varepsilon,\psi_j\right)_{\rho}=(\varepsilon,\partial_k\Phi_n)_{\rho}=0,\ 1\le j\le n+1,\ 1\le k\le 3.
\end{equation}
We know from \eqref{map} and \eqref{varepsilon} that
\begin{equation}\label{dec}
(\Phi_n+\bar{u})(\bar{\lambda} y+\tilde{x})=\frac{1}{\bar{\lambda}^{\frac{2}{p-1}}}(\Phi_n+\varepsilon+\sum_{j=2}^{n+1}a_j\psi_j)(y).
\end{equation}
For $u\in X_\delta$, combining \eqref{dec}, we have
$$
u=\frac{1}{\mu^{\frac{2}{p-1}}}(\Phi_n+\bar{u})\bigg(\frac{x-x'}{\mu}\bigg)=\frac{1}{\mu^{\frac{2}{p-1}}}\left[\frac{1}{\bar{\lambda}^{\frac{2}{p-1}}}(\Phi_n+\varepsilon+\sum_{j=2}^{n+1}a_j\psi_j)\right]\left(\frac{x-x'-\mu \tilde{x}}{\bar{\lambda} \mu}\right).
$$
Let $\lambda=\mu\bar{\lambda}$ and $\bar{x}=x'+\mu \tilde{x}$, then for any $u\in X_\delta$, the above decomposition can be written as 
\begin{equation}\label{varepsilon2}
u=\frac{1}{\lambda^{\frac{2}{p-1}}}(\Phi_n+\varepsilon+\sum_{j=2}^{n+1}a_j\psi_j)\left(\frac{x-\bar{x}}{\lambda}\right).
\end{equation}
In addition, from \eqref{Smooth-map} and $\|\bar u\|_{L^\infty}<\delta$, there exists a positive constant $C$ such that 
\begin{equation}\label{bound}
||\varepsilon||_{L^\infty}+\left|
\frac{\lambda}{\mu}-1\right|+\sum_{j=2}^{n+1}|a_j|+\frac{|\bar{x}-x'|}{\mu}\le C\delta.
\end{equation}
Combining \eqref{Smooth-map}, \eqref{varepsilon1}, \eqref{varepsilon2} and \eqref{bound}, we conclude the proof.
\end{proof}

\subsection{Choice of the initial data}
We next choose the initial data of \eqref{1.1} with the following form:
\begin{equation}\label{initial-data}
    u_0(x)=\frac{1}{\lambda_0^{\frac{2}{p-1}}}\left(\Phi_n+\varepsilon_0+\sum_{j=2}^{n+1}a_j\psi_j\right)\left(\frac{x-x_0}{\lambda_0}\right),
\end{equation}
where $x_0\in\mathbb{R}^3$ and
\begin{equation}\label{initial-setting1}
\left(\varepsilon_0,\psi_j\right)_{\rho}=(\varepsilon_0,\partial_k\Phi_n)_\rho=0,\ \ 1\le j\le n+1,\ \ 1\le k\le 3.
\end{equation}
Take $s_0\gg1$, $\mu$, $K_0$ three constants to be fixed later on, we restrict that the above parameters satisfying
\begin{equation}\label{initial-setting}
\lambda_0=\lambda(s_0)=e^{-s_0},\ \left\|\varepsilon_0\right\|_{L_\rho^2} < K_0 e^{-\mu s_0},\ \left\|\varepsilon_0\right\|_{L^\infty} < K_0 e^{-\mu s_0}.
\end{equation} and
\begin{equation}\label{initial-unstable-model}
    \sum_{j=2}^{n+1}|a_j|^2\le e^{-2\mu s_0}.
\end{equation}
%
%
%
%

\subsection{Bootstrap for the renormalized flow.}
As long as the solution $u(t)$ of \eqref{1.1} starting from \eqref{initial-data} belongs to $X_\delta$, we apply  Lemma \ref{Implicit} to deduce that the solution $u(t)$ can be written as
\begin{equation}\label{re}
u(t, x)=\frac{1}{\lambda(t)^\frac{2}{p-1}}\left(\Phi_n+v\right)(s, y),
\end{equation}
where $(s,y)$ denote the rescaled time and space variable, which are defined by
\begin{equation}\label{renormalized time}
s(t):=\int_0^t\frac{d\tau}{\lambda^2(\tau)}+s_0,\ \ y=\frac{x-x(t)}{\lambda(t)},
\end{equation}
and\begin{equation}\label{rts}
v=\varepsilon+\psi,\ \ \psi=\sum_{j=2}^{n+1}a_j(s)\psi_j,
\end{equation}
and
$\varepsilon$ is a function satisfying
\begin{equation}\label{orthogonality}
\left(\varepsilon,\psi_j\right)_{\rho}=(\varepsilon,\partial_k\Phi_n)_{\rho}=0,\ \ 1\le j\le n+1,\ \ 1\le k\le 3,
\end{equation}

From parabolic regularizing effects, the above decomposition is differentiable with respect to time.
Injecting \eqref{re} into \eqref{1.1} yields the renormalized equation
\begin{equation}\label{renormalized equation}
\partial_{s} \varepsilon+L_n \varepsilon=F+\operatorname{Mod},
\end{equation}
where
$$
\text{Mod}=\sum_{j=2}^{n+1}[\mu_ja_j-(a_j)_s]\psi_j+\left(\frac{\lambda_s}{\lambda}+1\right) \left(\Lambda\Phi_n+\Lambda\psi\right)+\frac{x_s}{\lambda}\cdot(\nabla\Phi_n+\nabla\psi),
$$
and  $F=\tilde{L}(\varepsilon)+NL$,
\begin{equation}\label{force}
\tilde{L}(\varepsilon)=\left(\frac{\lambda_s}{\lambda}+1\right) \Lambda\varepsilon+\frac{x_s}{\lambda}\cdot\nabla\varepsilon,\ \ NL=\lambda^{\frac{2p}{p-1}}f(\lambda y+x(t))+|\Phi_n+v|^{p-1}(\Phi_n+v)-\Phi_n^p-p\Phi_n^{p-1}v.
\end{equation}

We claim the following bootstrap proposition. Firstly,
we specify the order in which the parameters are chosen. Fix
$n>N$ as in Proposition 2.1. All constants below may depend on $n$ and $p$.
Let $c_n>0$ be the spectral gap constant in \eqref{spectral-gap}. We choose
\[
0<\mu<\frac{c_n}{4}
\]
sufficiently small. Then we choose
\[
0<K_0\ll1,\qquad K\gg K_0,\qquad K'\gg C(K,K_0),
\]
where $C(K,K_0)$ is the constant appearing in the Lemma \ref{Linfty}.
Finally, we choose
\[
s_0\ge s_*(n,p,\mu,K_0,K,K',\|f\|_{L^\infty})
\]
sufficiently large. Equivalently, $\lambda_0=e^{-s_0}$ is chosen sufficiently
small.
No smallness assumption is imposed on $\|f\|_{L^\infty}$ itself.

\begin{proposition}\label{bootstrap}
There exist constants $0<\mu,K_0\ll1$,  $K\gg1$ and $K'\gg1$ chosen as above such that the following holds.
For every $s_0\ge s_*(n,p, \mu, K_0,K,K',\|f\|_{L^\infty})\gg1$ sufficiently large, the following holds: Let $u_0$ satisfy \eqref{initial-data}, \eqref{initial-setting1} and \eqref{initial-setting}, then there exists $(a_{2},\cdots ,a_{n+1})$ satisfying \eqref{initial-unstable-model} such that the solution starting from $u_0$ can be decomposed according to \eqref{re} for all $s\ge s_0$:\\
- control of the scaling:
\begin{equation}\label{1.7}
0<\lambda(s)<e^{-\mu s};
\end{equation}
- control of the unstable modes:
\begin{equation}\label{chca}
\sum_{j=2}^{n+1}\left|a_j\right|^2 \le e^{-2\mu s};
\end{equation}
- control of the weighted norm:
\begin{equation}\label{1.9}
\left\|\varepsilon\right\|_{L_\rho^2} < K e^{-\mu s};
\end{equation}
- control of the maximum norm:
\begin{equation}\label{2.0}
\left\|\varepsilon\right\|_{L^\infty} < K' e^{-\mu s};
\end{equation}
\end{proposition}

The remainder of this paper is dedicated to proving Proposition \ref{bootstrap}, which directly implies Theorem \ref{thm}.
We next define the exit time
\begin{equation}\label{mmbb}
s^*=\sup\{s\ge s_0: \text{the bounds}\ \eqref{1.7}-\eqref{2.0}\ \text{holds on}\ [s_0,s) \},
\end{equation}
and we assume, by contradiction, that
\begin{equation}\label{sdzx}
s^*<+\infty.
\end{equation}
If we choose $K$, $K'$ and $s_0$ to be sufficiently large, then $s^*>s_0$. 

In the rest of this paper, we study the flow on $[s_0,s^*]$ where $\eqref{1.7}$-$\eqref{2.0}$ hold. We employ the bootstrap argument to show that the bounds $\eqref{1.7}$, $\eqref{1.9}$ and $\eqref{2.0}$ always hold on the time interval $[s_0,s^*]$. This implies that the unstable modes have grown, and \eqref{chca} holds with the equal sign in the exit time $s^*$.  However, this contradicts the Brouwer fixed-point theorem. Hence, $s^* = +\infty$.
\subsection{Modulation equation}
\begin{lemma}\label{modulation}
For the solution on $[s_0,s^*]$, if we take $s_0$ large enough, we have the following estimate on the modulation parameters:
\begin{equation}\label{semin}
\bigg|\frac{\lambda_s}{\lambda}+1\bigg|+\sum_{j=2}^{n+1}|(a_j)_s-\mu_ja_j|+\left|\frac{x_s}{\lambda}\right|\lesssim||\varepsilon||^2_{L^2_\rho}+\sum_{j=2}^{n+1}|a_j|^2+|\lambda(s)|^{\frac{2p}{p-1}}.
\end{equation}
\end{lemma}
\begin{proof}
\emph{$\mathbf{Step\ 1.}$}\ Estimate for $|(a_j)_s-\mu_ja_j|$. Take the $L^2_\rho$ scalar product of \eqref{renormalized equation} with $\psi_j$ for $2\le j\le n+1$, from \eqref{orthogonality} and using the fact that the operator $L_n$  is  self-adjoint for the $L^2_{\rho}$ product, and by \eqref{nonbus},
 we have $$-(\text{Mod},\psi_j)_\rho=(F,\psi_j)_\rho,$$ i.e.,
\begin{equation}\label{modul-1}
(a_j)_s-\mu_ja_j=(F,\psi_j)_\rho+\left(\frac{\lambda_s}{\lambda}+1\right)(\Lambda \psi,\psi_j)_\rho+\left(\frac{x_s}{\lambda}\cdot(\nabla\Phi_n+\nabla\psi),\psi_j\right)_\rho.
\end{equation}
We know from Proposition \ref{prop:spectral-gap} that
\begin{equation}\label{chacs}
|(\Lambda \psi,\psi_j)_\rho|\lesssim\sum_{j=2}^{n+1}|a_j|.
\end{equation}
From the orthogonality conditions $$(\partial_k\Phi_n,\psi_j)_\rho=0,\ 1\le k\le 3,\ 2\le j\le n+1,$$ we have 
\begin{equation}\label{modul-2}
\left|\left(\frac{x_s}{\lambda}\cdot(\nabla\Phi_n+\nabla\psi),\psi_j\right)_\rho\right|=\left|\left(\frac{x_s}{\lambda}\cdot\nabla\psi,\psi_j\right)_\rho\right|
\lesssim \sum_{j=2}^{n+1}|a_j|\left|\frac{x_s}{\lambda}\right|.
\end{equation}

We next estimate the term $(F,\psi_j)_\rho$. By Cauchy-Schwarz inequality and integrating by parts, we get
\begin{equation}\label{dells}
\begin{aligned}
|(\tilde{L}(\varepsilon),\psi_j)_\rho|=\left|\left(\left(\frac{\lambda_s}{\lambda}+1\right)\Lambda\varepsilon+\frac{x_s}{\lambda}\cdot\nabla\varepsilon,\psi_j\right)_\rho\right|
\lesssim \left(\left|\frac{\lambda_s}{\lambda}+1\right|+\left|\frac{x_s}{\lambda}\right|\right)||\varepsilon||_{L^2_\rho}.
\end{aligned}
\end{equation}
By $p>5$ and the smallness of $v$, from the Taylor expansion, we have 
\begin{equation}\label{polyonmial-estiamate}
\left||\Phi_n+v|^{p-1}(\Phi_n+v)-\Phi_n^p-p\Phi_n^{p-1}v\right|\lesssim v^2+v^p\lesssim v^2.    
\end{equation}
Thus we have
\begin{equation}\label{NL-estimate}
\begin{aligned}
&\left|(|\Phi_n+v|^{p-1}(\Phi_n+v)-\Phi_n^p-p\Phi_n^{p-1}v,\psi_j)_\rho\right|
\\
&\lesssim \left|(v^2,\psi_j)_\rho\right|\lesssim\left|(\varepsilon^2,\psi_j)_\rho\right|+\sum_{j=2}^{n+1}|a_j|^2\left|(\sum_{j=2}^{n+1}\psi_j^2,\psi_j)_\rho\right|\lesssim ||\varepsilon||_{L^2_\rho}^2+\sum_{j=2}^{n+1}|a_j|^2.
\end{aligned}
\end{equation}
By  \eqref{NL-estimate} and the boundedness of $f$, we obtain
\begin{equation}\label{dellss}
\begin{aligned}
|(NL,\psi_j)_\rho|&\le(\left||\Phi_n+v|^{p-1}(\Phi_n+v)-\Phi_n^p-p\Phi_n^{p-1}v\right|,\psi_j)_\rho+|\lambda|^{\frac{2p}{p-1}}\left(\left|f\right|,\psi_j\right)_\rho\\
&\lesssim  ||\varepsilon||^2_{L^2_\rho}+\sum_{j=2}^{n+1}|a_j|^2+|\lambda|^{\frac{2p}{p-1}}.
\end{aligned}
\end{equation}
Combining \eqref{modul-1}-\eqref{dells} and \eqref{dellss} yields
\begin{equation}\label{cafss}
|(a_j)_s-\mu_ja_j|\lesssim \left|\frac{\lambda_s}{\lambda}+1\right|\left(\sum_{j=2}^{n+1}|a_j|+||\varepsilon||_{L^2_\rho}\right)+\left|\frac{x_s}{\lambda}\right|||\varepsilon||_{L^2_\rho}+\sum_{j=2}^{n+1}|a_j|\left|\frac{x_s}{\lambda}\right|+||\varepsilon||^2_{L^2_\rho}+\sum_{j=2}^{n+1}|a_j|^2+|\lambda|^{\frac{2p}{p-1}}.
\end{equation}
\emph{$\mathbf{Step\ 2.}$}\ Estimate for $\left|\frac{\lambda_s}{\lambda}+1\right|$.
Take the $L^2_\rho$ scalar product of \eqref{renormalized equation} with $\psi_1=\frac{\Lambda \Phi_n}{||\Lambda \Phi_n||^2_{L^2_{\rho}}}$, we get
$$
-\bigg(\frac{\lambda_s}{\lambda}+1\bigg)=(F,\psi_1)_\rho+\bigg(\frac{\lambda_s}{\lambda}+1\bigg)(\Lambda\psi,\psi_1)_\rho+\left(\frac{x_s}{\lambda}\cdot\nabla\psi,\psi_1\right)_\rho.
$$
Then by a same way as in Step 1, we get
\begin{equation}\label{cafss0}
\left|\frac{\lambda_s}{\lambda}+1\right|\lesssim \left|\frac{\lambda_s}{\lambda}+1\right|\left(\sum_{j=2}^{n+1}|a_j|+||\varepsilon||_{L^2_\rho}\right)+||\varepsilon||^2_{L^2_\rho}+\sum_{j=2}^{n+1}|a_j|^2+\left|\frac{x_s}{\lambda}\right|||\varepsilon||_{L^2_\rho}+\sum_{j=2}^{n+1}|a_j|\left|\frac{x_s}{\lambda}\right|+|\lambda|^{\frac{2p}{p-1}}.
\end{equation}
\emph{$\mathbf{Step\ 3.}$}\ Estimate for $\left|\frac{x_s}{\lambda}\right|$. Take the $L^2_\rho$ scalar product of \eqref{renormalized equation} with $\frac{\partial_1 \Phi_n}{||\partial_1 \Phi_n||^2_{L^2_{\rho}}}$, we get
$$
\frac{-x_s}{\lambda}=\bigg(\frac{\lambda_s}{\lambda}+1\bigg)\left(\Lambda\psi,\frac{\partial_1 \Phi_n}{||\partial_1 \Phi_n||^2_{L^2_{\rho}}}\right)_\rho+\left(\frac{x_s}{\lambda}\cdot\nabla\psi,\frac{\partial_1 \Phi_n}{||\partial_1 \Phi_n||^2_{L^2_{\rho}}}\right)_\rho+\left(F,\frac{\partial_1 \Phi_n}{||\partial_1 \Phi_n||^2_{L^2_{\rho}}}\right)_\rho,
$$
By a same way as in Step 1, we obtain
\begin{equation}\label{modul-last}
\left|\frac{x_s}{\lambda}\right|\lesssim \left|\frac{\lambda_s}{\lambda}+1\right|\left(\sum_{j=2}^{n+1}|a_j|+||\varepsilon||_{L^2_\rho}\right)+||\varepsilon||^2_{L^2_\rho}+\sum_{j=2}^{n+1}|a_j|^2+\left|\frac{x_s}{\lambda}\right|||\varepsilon||_{L^2_\rho}+\sum_{j=2}^{n+1}|a_j|\left|\frac{x_s}{\lambda}\right|+|\lambda|^{\frac{2p}{p-1}}.
\end{equation}
\emph{$\mathbf{Step\ 4.}$}\ Choose $s_0$ sufficiently large.
We know from \eqref{cafss}, \eqref{cafss0} and \eqref{modul-last} that
\begin{equation}
\begin{aligned}
&\sum_{j=2}^{n+1}|(a_j)_s-\mu_ja_j|+\left|\frac{\lambda_s}{\lambda}+1\right|+\left|\frac{x_s}{\lambda}\right|\\
&\lesssim \left|\frac{\lambda_s}{\lambda}+1\right|\left(\sum_{j=2}^{n+1}|a_j|+||\varepsilon||_{L^2_\rho}\right)+||\varepsilon||^2_{L^2_\rho}+\sum_{j=2}^{n+1}|a_j|^2+\left|\frac{x_s}{\lambda}\right|||\varepsilon||_{L^2_\rho}+\sum_{j=2}^{n+1}|a_j|\left|\frac{x_s}{\lambda}\right|+|\lambda|^{\frac{2p}{p-1}}.
\end{aligned}
\end{equation}
Then we take $s_0$ large enough, by \eqref{1.9} and \eqref{chca}, we have
\begin{equation}
\begin{aligned}
\sum_{j=2}^{n+1}|(a_j)_s-\mu_ja_j|+\left|\frac{\lambda_s}{\lambda}+1\right|+\left|\frac{x_s}{\lambda}\right|\lesssim||\varepsilon||^2_{L^2_\rho}+\sum_{j=2}^{n+1}|a_j|^2+|\lambda|^{\frac{2p}{p-1}}.
\end{aligned}
\end{equation}
This completes the proof.
\end{proof}

\subsection{Energy estimates with exponential weights}
\begin{lemma}
For $s_0$ large enough, there holds the bound
\begin{equation}\label{Eys}
\frac{d}{ds}||\varepsilon||_{L^2_\rho}^2+c_n||\varepsilon||_{H^1_\rho}^2\lesssim \sum_{j=2}^{n+1}|a_j|^4+|\lambda|^{\frac{4p}{p-1}},
\end{equation}
where the constant $c_n>0$ is given by \eqref{spectral-gap}.
\end{lemma}

\begin{proof}
Take the $L^2_\rho$ scalar product of \eqref{renormalized equation} with $\varepsilon$, we have
\begin{equation}\label{xl}
\frac{1}{2}\frac{d}{ds}||\varepsilon||_{L^2_\rho}^2=-(L_n\varepsilon,\varepsilon)_\rho+(F+\text{Mod},\varepsilon)_\rho.
\end{equation}
Combining \eqref{spectral-gap} and the orthogonality condition \eqref{orthogonality} yields
\begin{equation}\label{gaps}
-(L_n\varepsilon,\varepsilon)_\rho\le -c_n||\varepsilon||_{H^1_\rho}^2.
\end{equation}
Using Cauchy's inequality with $\delta$:
\begin{equation}\label{Cauchy inequality}
ab\le \delta a^2+C_\delta b^2,\ \ \left(a,b>0,\ \delta>0,\ C_\delta=\frac{1}{4\delta}\right),
\end{equation}
combining \eqref{chca}, \eqref{semin} and Cauchy-Schwarz inequality, we get
\begin{equation}\label{lyya}
\begin{aligned}
(\text{Mod},\varepsilon)_\rho&\le||\varepsilon||_{L^2_\rho}||\text{Mod}||_{L^2_\rho}\lesssim||\varepsilon||_{L^2_\rho}\left(||\varepsilon||^2_{L^2_\rho}+\sum_{j=2}^{n+1}|a_j|^2+|\lambda|^{\frac{2p}{p-1}}\right)
\\
&\lesssim \delta||\varepsilon||_{L^2_\rho}^2+C_\delta\left(||\varepsilon||^4_{L^2_\rho}+\sum_{j=2}^{n+1}|a_j|^4+|\lambda|^{\frac{4p}{p-1}}\right).
\end{aligned}
\end{equation}
By integrating by parts, Lemma A.1 in \cite{Collot-Pierre-2019} and Cauchy-Schwarz inequality, we obtain
\begin{equation}\label{lyysd}
\begin{aligned}
|(\tilde{L}(\varepsilon),\varepsilon)_\rho|&\lesssim\left|\frac{\lambda_s}{\lambda}+1\right|\left(\int_{\mathbb{R}^3}\varepsilon^2\rho dy+\left|\int_{\mathbb{R}^3}\varepsilon\nabla\cdot(y\varepsilon\rho) dy\right|\right)+\left|\int_{\mathbb{R}^3}\varepsilon\frac{x_s}{\lambda}\cdot\nabla(\varepsilon\rho)dy\right|\\
&\lesssim \left(\left|\frac{\lambda_s}{\lambda}+1\right|+\left|\frac{x_s}{\lambda}\right|\right)||\varepsilon||_{H^1_\rho}^2.
\end{aligned}
\end{equation}
By Cauchy-Schwarz inequality and integrating by parts, we deduce from \eqref{polyonmial-estiamate} and \eqref{Cauchy inequality} that
\begin{equation}\label{xl1}
\begin{aligned}
|(NL,\varepsilon)_\rho|&\lesssim
\left|\int_{\mathbb{R}^3}(\varepsilon^2+\psi^2)\varepsilon\rho dy\right|+\left|\lambda\right|^{\frac{2p}{p-1}}\left|\int_{\mathbb{R}^3}f(\lambda y+x(t))\varepsilon\rho dy\right|\\
&\lesssim ||\varepsilon||_{L^\infty}||\varepsilon||^2_{H^1_\rho}+\sum^{n+1}_{j=2}|a_j|^2||\varepsilon||_{H^1_\rho}+\left|\lambda\right|^{\frac{2p}{p-1}}||f||_\infty\left(\int_{\mathbb{R}^3}\varepsilon^2\rho dy\right)^{\frac{1}{2}}\left(\int_{\mathbb{R}^3}\rho dy\right)^{\frac{1}{2}}
\\
&\lesssim\bigg(||\varepsilon||_{L^\infty}+2\delta\bigg)||\varepsilon||^2_{H^1_\rho}+C_\delta\left(\sum^N_{j=2}|a_j|^4+\left|\lambda\right|^{\frac{4p}{p-1}}\right).
\end{aligned}
\end{equation}

From \eqref{xl}, \eqref{gaps}, \eqref{lyya}, \eqref{lyysd} and \eqref{xl1},   if we choose $s_0$ sufficiently large and $\delta$ small enough, we then get the desired estimate by \eqref{1.7}-\eqref{2.0} and \eqref{semin}.
\end{proof}

\subsection{$L^\infty$ bound of $\varepsilon$}
\begin{lemma}\label{Linfty}
For $ s_0$ large enough, there holds the bound
\begin{equation}\label{dfyusd}
||\varepsilon||_{L^\infty}\le C(K,K_0)e^{-\mu s}.
\end{equation}
\end{lemma}

\begin{proof}
We rewrite the equation \eqref{renormalized equation} in the following form:
\begin{equation}\nonumber
\partial_s\varepsilon=L\varepsilon+G,
\end{equation}
where $L\varepsilon=\Delta \varepsilon+T\cdot\nabla \varepsilon+V\varepsilon$ and
$$
T=\frac{\lambda_sy}{\lambda}+\frac{x_s}{\lambda},\ \ V=\left(\frac{\lambda_s}{\lambda}+1\right)\frac{2}{p-1}+p\Phi_n^{p-1}-\frac{2}{p-1},$$
$$
G=|\Phi_n+v|^{p-1}(\Phi_n+v)-\Phi_n^p-p\Phi_n^{p-1}v+\lambda^{\frac{2p}{p-1}}f(\lambda y+x(t))+\text{Mod}.
$$
\emph{$\mathbf{Step\ 1}.$}\ $L^\infty$ estimate inside the ball. For arbitrary $R$ satisfying $0<R<\infty$, we know from \eqref{1.7}, \eqref{chca}, \eqref{1.9}, \eqref{semin} and Proposition \ref{self-similar profiles} that
\begin{equation}\label{fontaines}
||T(y)||_{L^\infty(B_{2R})}\lesssim1,\ \ ||V(y)||_{L^\infty(B_{2R})}\lesssim1.
\end{equation}
 Combining \eqref{polyonmial-estiamate}, $2p/(p-1)>1$ and the boundedness of the inhomogeneous term $f$, 
we obtain similarly that
\begin{equation}\label{fontaines0}
||G||_{L^\infty(B_{2R})}\le C(K)e^{-\mu s},
\end{equation}
for $ s_0$ large enough. We know from \eqref{1.9} that
\begin{equation}\label{jjyx}
||\varepsilon(s)||_{L^2(B_{2R})}\le C(K)e^{-\mu s},\ \ s\in[s_0,s^*].
\end{equation}
Combining \eqref{initial-setting}, \eqref{fontaines}, \eqref{fontaines0} and \eqref{jjyx}, by the initial \(L^\infty\) smallness and the parabolic maximum estimate,
we employ parabolic regularity (see, e.g., Theorem 48.1 in \cite{Quittner-Souplet-2019}) to obtain that, there exists a constant $C'=C'(K,K_0)$ such that
\begin{equation}\label{estimate-inside}
||\varepsilon(s)||_{L^\infty(B_R)}\le \frac{1}{2}C'e^{-\mu s}\ \ s\in[s_0,s^*].
\end{equation}
\emph{$\mathbf{Step\ 2}.$}\ $L^\infty$ bound outside the ball.
 We claim that
\begin{equation}\label{estimate-outside}
 ||\varepsilon||_{L^\infty(|y|\ge R)}\le C'e^{-\mu s},\ \ s\in[s_0,s^*].
\end{equation}
We next use parabolic comparison principle on $\{|y|\ge R\}\times [s_0,s^*]$ to prove the above claim.

Take space homogeneous function $\bar{\psi}(y,s)=C'e^{-\mu s}$. Since $K_0$ is small enough, by \eqref{initial-setting} and \eqref{estimate-inside}, on the boundary, there holds
\begin{equation}\label{dfyus}
\bar{\psi}(s_0,y)\ge |\varepsilon(s_0,y)|,\ \ |y|\ge R,
\end{equation}
and
\begin{equation}\label{qdu}
\bar{\psi}(s,R)\ge |\varepsilon(s,R)|,\ \ s\in[s_0,s^*].
\end{equation}
We compute
\begin{equation}\label{com1}
\begin{aligned}
\partial_s(\bar{\psi}-\varepsilon)-L(\bar{\psi}-\varepsilon)=\bar{\psi}_s-L\bar{\psi}-G,\ |y|\ge R.
\end{aligned}
\end{equation}
If we take $R$ and $s_0$ large enough, since $\mu>0$ is small enough, then by \eqref{large-r}, \eqref{1.7}, \eqref{chca}, \eqref{1.9}, and \eqref{semin},
we have
\begin{equation}\label{com}
\begin{aligned}
\bar{\psi}_s-L\bar{\psi}&=
\bigg[\frac{2}{p-1}-\mu-\left(\frac{\lambda_s}{\lambda}+1\right)\frac{2}{p-1}-p\Phi_n^{p-1}\bigg]\bar{\psi}\ge\frac{\bar{\psi}}{p-1}.
\end{aligned}
\end{equation}
Take $s_0$  large enough, by $p>5$ and the boundedness of $f$, combining \eqref{1.7}, \eqref{chca}, \eqref{1.9}, \eqref{2.0}, \eqref{semin}  and \eqref{polyonmial-estiamate} yields
\begin{equation}\label{com2}
\begin{aligned}
|G|\lesssim v^2+|\lambda|^{\frac{2p}{p-1}}||f||_{L^\infty}+|\text{Mod}|\le \frac{\bar{\psi}}{p-1}.
\end{aligned}
\end{equation}
From \eqref{com1}, \eqref{com} and \eqref{com2}, we have
\begin{equation}\nonumber
\begin{aligned}
\partial_s(\bar{\psi}-\varepsilon)-{L}(\bar{\psi}-\varepsilon)\ge0\ \ \text{on}\ \{|y|\ge R\}\times [s_0,s^*].
\end{aligned}
\end{equation}
Then combining the boundary conditions \eqref{dfyus} and \eqref{qdu},  we employ parabolic comparison principle to get
$$
\varepsilon(s,y)\le \bar{\psi}\ \ \text{on}\ \{|y|\ge R\}\times [s_0,s^*].
$$
We can show similarly that
$$
-\bar{\psi}\le \varepsilon(s,y)\ \ \text{on}\ \{|y|\ge R\}\times [s_0,s^*],
$$
which proves the claim.
 We complete the proof by \eqref{estimate-inside} and \eqref{estimate-outside}.
\end{proof}
We choose sufficiently large $K'$ satisfying $K'\gg C(K,K_0)$, then from \eqref{dfyusd}, we get
\begin{equation}\label{dfyusdA}
||\varepsilon(s)||_{L^\infty}\le {K'}e^{-\mu s},\
 \ s\in[s_0,s^*].
\end{equation}

\subsection{Conclusion}
We are now in the position to conclude the proof of Proposition \ref{bootstrap}.
\begin{proof}[Proof of Proposition \ref{bootstrap}]
Let us first recall the exit time
$$
s^*=\sup\{s\ge s_0\ \text{such that}\ \eqref{1.7}-\eqref{2.0}\ \text{holds on}\ [s_0,s) \}.
$$
\emph{$\mathbf{Step\ 1}.$}\ \emph{Improved scaling control.} From $\frac{2p}{p-1}>2$, by \eqref{1.7}, \eqref{chca},  \eqref{1.9} and \eqref{semin}, we get
\begin{equation}\label{fky}
\left|\frac{\lambda_s}{\lambda}+1\right|\lesssim K^2e^{-2\mu s}.
\end{equation}
Integrating both sides from \(s_0\) to \(s\), we obtain
\[
\left|
\log \lambda(s)-\log \lambda(s_0)+s-s_0
\right|
\lesssim
K^2\int_{s_0}^s e^{-2\mu\tau}\,d\tau
\le
\frac{K^2}{2\mu}e^{-2\mu s_0}.
\]
Since \(\lambda(s_0)=e^{-s_0}\), it follows that
\[
\log \lambda(s)
=
-s+O(e^{-2\mu s_0}).
\]
Exponentiating the above identity yields
\[
\lambda(s)
=
e^{-s}e^{O(e^{-2\mu s_0})}
=
e^{-s}\left(1+O(e^{-2\mu s_0})\right).
\]
In particular, for \(s_0\) sufficiently large,
\begin{equation}\label{poaf}
   \frac12 e^{-s}\le \lambda(s)\le 2e^{-s}. 
\end{equation}

\emph{$\mathbf{Step\ 2}.$}\ \emph{Improved $L^2_\rho$ bound.}
 By $\frac{p}{p-1}>1$, for $ s_0$ large enough, we know from \eqref{1.7}, \eqref{chca} and \eqref{Eys}  that
\begin{equation}\label{bfa}
\begin{aligned}
\frac{d}{ds}||\varepsilon||_{L^2_\rho}^2+c_n||\varepsilon||_{H^1_\rho}^2\lesssim e^{-4\mu s}+e^{\frac{-4p\mu s}{p-1}}\lesssim e^{-4\mu s}.
\end{aligned}
\end{equation}
 We next set
$
0<\mu\le\frac{c_n}{4}.
$
From \eqref{bfa}, we get
$$
\frac{d}{ds}||\varepsilon||_{L^2_\rho}^2\lesssim  e^{-4\mu s}-2\mu||\varepsilon||_{L^2_\rho}^2,
$$
then by Gronwall inequality, from \eqref{initial-setting}, for $s_0$ large enough, there holds
\begin{equation}\nonumber
\begin{aligned}
||\varepsilon(s)||_{L^2_\rho}^2
\lesssim e^{-2\mu (s-s_0)} \left(||\varepsilon(s_0)||_{L^2_\rho}^2+\int_{s_0}^se^{-4\mu \tau}d\tau\right)\lesssim K_0^2e^{-2\mu s}.
\end{aligned}
\end{equation}
Then, we take $K$ large enough such that
\begin{equation}\label{iiz}
\begin{aligned}
||\varepsilon(s)||_{L^2_\rho}
\le \frac{K}{2}e^{-\mu s}.
\end{aligned}
\end{equation}
\emph{$\mathbf{Step\ 3}$.}\ \emph{Topological argument.}
Set
\[
A(s)=\left(a_j(s)e^{\mu s}\right)_{2\le j\le n+1}\in\mathbb R^n,
\]
and let \(\mathcal B\) be the closed unit ball in \(\mathbb R^n\). Combining the definitions of $s^*$ and the contradiction assumption \eqref{sdzx}, along with \eqref{dfyusdA}, \eqref{poaf} and \eqref{iiz}, then through a continuity argument and \eqref{chca}, the exit
occurs through the unstable modes, namely
\[
|A(s^*)|^2=\sum_{j=2}^{n+1}|a_j(s^*)|^2e^{2\mu s^*}=1.
\]
Moreover, from the modulation equations we have
\[
\begin{aligned}
\frac12\frac{d}{ds}|A(s)|^2
&=
\sum_{j=2}^{n+1}
a_j e^{2\mu s}
\left[
(\mu_j+\mu)a_j+\big((a_j)_s-\mu_j a_j\big)
\right]  \\
&\ge
\mu e^{2\mu s}\sum_{j=2}^{n+1}|a_j|^2
-
Ce^{-\mu s}.
\end{aligned}
\]
Therefore, at \(s=s^*\),
\[
\left.\frac12\frac{d}{ds}|A(s)|^2\right|_{s=s^*}
\ge
\mu-Ce^{-\mu s^*}>0
\]
provided \(s_0\) is sufficiently large. Hence the vector field is strictly
outgoing at the exit boundary.

By contradiction assumption, we know that \(s^*<\infty\) for every initial choice
\(A(s_0)\in\mathcal B\). Then we may define the exit map
\[
\Gamma:\mathcal B\to\partial\mathcal B,\qquad
\Gamma(A(s_0))=A(s^*).
\]
The continuous dependence of the flow on the initial data, together with the
strictly outgoing property, implies that \(\Gamma\) is continuous. Moreover,
if \(A(s_0)\in\partial\mathcal B\), then \(s^*=s_0\) and hence
$\Gamma(A(s_0))=A(s_0)$.
Thus \(\Gamma\) is a continuous retraction from \(\mathcal B\) onto
\(\partial\mathcal B\).
Equivalently, the map \(-\Gamma:\mathcal B\to\mathcal B\) is continuous.
By the Brouwer fixed point theorem, there exists \(A_0\in\mathcal B\) such
that
$A_0=-\Gamma(A_0)$.
Since \(\Gamma(A_0)\in\partial\mathcal B\), we have \(A_0\in\partial\mathcal B\).
But on \(\partial\mathcal B\), \(\Gamma\) is the identity. Hence
$A_0=-A_0,$
which implies \(A_0=0\), contradicting \(A_0\in\partial\mathcal B\). Therefore
there exists an initial parameter \(A(s_0)\in\mathcal B\) such that the
corresponding solution does not exit the bootstrap regime. This concludes the proof of Proposition \ref{bootstrap}.
\end{proof}

Next we present the proof of Theorem \ref{thm}.
\begin{proof}[Proof of Theorem \ref{thm} ]
\emph{$\mathbf{Step\ 1}.$} \emph{Finite time blow-up}.
From \eqref{renormalized time}, we get $\frac{dt}{ds}=\lambda^2(t(s))=\lambda^2(s)$, then by
\eqref{poaf}, we have
$$
\int_{s_0}^{+\infty}\frac{1}{4}e^{-2s}ds\le T=\int_{s_0}^{+\infty}\lambda^2(s)ds\le\int_{s_0}^{+\infty}4e^{-2s}ds,
$$
then by $\lambda_0=\lambda(s_0)=e^{-s_0}$, there holds $$\frac{\lambda_0^2}{8}\le T\le 2\lambda_0^2.$$
In addition, by
\eqref{poaf}, we have
\begin{equation}\label{times}
\frac{e^{-2s}}{8}\le T-t=\int_s^\infty\lambda^2(s)ds\le 2e^{-2s}.
\end{equation}
By $\lambda_s=\lambda_tt_s=\lambda^2\lambda_t$, combining
\eqref{fky}, we get 
\begin{equation}\label{pams}
 \left|\lambda\lambda_t+1\right|\lesssim (T-t)^{\mu}.
\end{equation}
 Then integrating \eqref{pams} from $t$ to $T$ and using the boundary condition $\lambda(T)=0$, we obtain the asymptotic behavior
\begin{equation}\label{time0s}
\lambda(t)=\sqrt{2(T-t)}(1+o(1)),\ \ t\to T.
\end{equation}
From \eqref{renormalized time}, \eqref{semin} and \eqref{poaf}, we have
$$
\int_0^T\left|{x_t}\right|dt=\int_{s_0}^\infty\left|{x_s}\right|ds\lesssim \int_{s_0}^\infty e^{-(1+2\mu) s}ds<\infty,
$$
Thus, \eqref{blow-up point} is proved.

\emph{$\mathbf{Step\ 2.}$}\ \emph{Asymptotic stability of the  profiles $\Phi_n$ above scaling.}
Since the eigenvectors $\psi_j$ are bounded, then by \eqref{chca} and \eqref{2.0} we have
$
|v|\lesssim e^{-\mu s},
$
then combining \eqref{times} yields
\begin{equation}\label{frsv}
||v(t)||_{L^\infty}\lesssim(T-t)^\frac{\mu}{2}\to 0,\ \ \text{as}\ t\to T.
\end{equation}
From \eqref{re}, we have
\begin{equation}\label{js}
\begin{aligned}
u(t,x)&=\frac{1}{[2(T-t)]^{\frac{1}{p-1}}}\Phi_n\left(\frac{x-x(t)}{\sqrt{2(T-t)}}\right)+\frac{1}{\lambda(t)^\frac{2}{p-1}}\Phi_n\left(\frac{x-x(t)}{\lambda(t)}\right)\\
&\ \ \ \  -\frac{1}{[2(T-t)]^{\frac{1}{p-1}}}\Phi_n\left(\frac{x-x(t)}{\sqrt{2(T-t)}}\right)+\frac{1}{\lambda(t)^\frac{2}{p-1}}v(s,y)\\
&=\frac{1}{[2(T-t)]^{\frac{1}{p-1}}}\left(\Phi_n+\bar{u}\right)\left(t,\frac{x-x(t)}{\sqrt{2(T-t)}}\right),
\end{aligned}
\end{equation}
where $$\bar{u}:=[2(T-t)]^{\frac{1}{p-1}}\left[\frac{1}{\lambda(t)^\frac{2}{p-1}}\Phi_n\left(\frac{x-x(t)}{\lambda(t)}\right) -\frac{1}{[2(T-t)]^{\frac{1}{p-1}}}\Phi_n\left(\frac{x-x(t)}{\sqrt{2(T-t)}}\right)+\frac{1}{\lambda(t)^\frac{2}{p-1}}v(s,y)\right].$$
Combining \eqref{time0s} and \eqref{frsv}, we get
\begin{equation}\label{reminder term}
  \lim_{t\to T}||\bar{u}(t)||_{L^\infty(\mathbb{R}^3)}=0,   
\end{equation}
and
\eqref{stability} is proved. 

We next show that \(x(T)\) is a blow-up point. By \eqref{js} and \eqref{reminder term}, we have
\[
|u(t,x(t))|
=
\frac{1}{[2(T-t)]^{\frac{1}{p-1}}}\left|\Phi_n(0)+\bar{u}(t,0)\right|
\to+\infty
\]
as \(t\to T^-\). Hence, since \(x(t)\to x(T)\), the point \(x(T)\) is a
blow-up point.
\end{proof}
We next prove the
Lipschitz dependence of the constructed set of solutions in this paper.
\begin{proposition}
\label{prop:lipschitz-source}
Assume that $
f\in L^\infty\cap C^{0,1}(\mathbb R^3)$. Let
\[
\bar u_0^{(1)},\bar u_0^{(2)}
\in V_n^\perp\cap B_{L^\infty}(0,\delta)
\]
be two stable initial perturbations. 
Let \(s_0\) be sufficiently large. Then the parameters $(a_j^{(1)}(s_0)_{2\le j\le n+1}$ and  $(a_j^{(2)}(s_0)_{2\le j\le n+1}$  given by Proposition \ref{bootstrap} satisfy
\[
\sum_{j=2}^{n+1}
\left|
a_j^{(1)}(s_0)-a_j^{(2)}(s_0)
\right|^2
\lesssim
\left\|
\bar u_0^{(1)}-\bar u_0^{(2)}
\right\|_{L^2_\rho}^2.
\]
In particular,
\[
\sum_{j=2}^{n+1}
\left|
a_j^{(1)}(s_0)-a_j^{(2)}(s_0)
\right|
\lesssim
\left\|
\bar u_0^{(1)}-\bar u_0^{(2)}
\right\|_{L^\infty}.
\]
Consequently, the map
\[
\bar u_0\mapsto
\big(a_2(\bar u_0),\ldots,a_{n+1}(\bar u_0)\big)
\]
is Lipschitz. Moreover, the corresponding blow-up times satisfy
\[
\left|T^{(1)}-T^{(2)}\right|
\lesssim
\left\|
\bar u_0^{(1)}-\bar u_0^{(2)}
\right\|_{L^\infty}.
\]
\end{proposition}

\begin{proof}
Let
\[
A^{(i)}
=
\big(a_2^{(i)}(s_0),\ldots,a_{n+1}^{(i)}(s_0)\big),
\qquad i=1,2,
\]
be two choices of unstable parameters given by Proposition
\ref{bootstrap}. Denote by $
\varepsilon^{(i)}$,\  $a_j^{(i)}$,\ $\lambda^{(i)}$,\ $x^{(i)}$
the corresponding renormalized solutions and modulation parameters. We set
\[
v^{(i)}
=
\varepsilon^{(i)}
+
\psi^{(i)},
\qquad \psi^{(i)}=\sum_{j=2}^{n+1}a_j^{(i)}\psi_j,
\qquad
\beta:=\frac{2p}{p-1}.
\]
Define the differences
\[
\Delta\varepsilon
:=
\varepsilon^{(1)}-\varepsilon^{(2)},\qquad
\Delta a_j
:=
a_j^{(1)}-a_j^{(2)},\qquad
\Delta v
:=
v^{(1)}-v^{(2)},
\]
and
\[
\Delta \ell
:=
\log\frac{\lambda^{(1)}}{\lambda^{(2)}},
\qquad
\Delta x
:=
x^{(1)}-x^{(2)}.
\]

We compare the two renormalized flows at the same renormalized time \(s\).
Recall that each \(\varepsilon^{(i)}\) satisfies
\[
\partial_s\varepsilon^{(i)}+L_n\varepsilon^{(i)}
=
\widetilde L^{(i)}(\varepsilon^{(i)})
+
\text{NL}^{(i)}
+
\text{Mod}^{(i)},
\]
where
\[
\widetilde L^{(i)}(\varepsilon^{(i)})
=
\left(
\frac{\lambda_s^{(i)}}{\lambda^{(i)}}+1
\right)
\Lambda\varepsilon^{(i)}
+
\frac{x_s^{(i)}}{\lambda^{(i)}}
\cdot\nabla\varepsilon^{(i)}
,\]
$$\text{Mod}^{(i)}=\sum_{j=2}^{n+1}[\mu_ja_j^{(i)}-(a_j)^{(i)}_s]\psi_j+\left(\frac{\lambda^{(i)}_s}{\lambda^{(i)}}+1\right) \left(\Lambda\Phi_n+\Lambda\psi^{(i)}\right)+\frac{x^{(i)}_s}{\lambda^{(i)}}\cdot(\nabla\Phi_n+\nabla\psi^{(i)}),$$
and
\[
\begin{aligned}
\text{NL}^{(i)}
&=
\left(\lambda^{(i)}\right)^\beta
f\left(\lambda^{(i)}y+x^{(i)}\right)
+
|\Phi_n+v^{(i)}|^{p-1}(\Phi_n+v^{(i)})
-\Phi_n^p
-p\Phi_n^{p-1}v^{(i)}.
\end{aligned}
\]
Subtracting the two equations yields
\begin{equation}\label{eq:diff-eps-source}
\partial_s\Delta\varepsilon
+
L_n\Delta\varepsilon
=
\Delta\widetilde L
+
\Delta \text{NL}
+
\Delta \text{Mod}.
\end{equation}

\medskip

\noindent
\emph{Step 1. Estimate of the difference of the source terms.}
Set
\[
\mathcal F^{(i)}(s,y)
=
\left(\lambda^{(i)}(s)\right)^\beta
f\left(\lambda^{(i)}(s)y+x^{(i)}(s)\right).
\]
Since \(f\in L^\infty\cap C^{0,1}\) and $
\lambda^{(i)}(s)\sim e^{-s},$
we have
\begin{equation}\label{eq:source-diff-Lip}
\left|
\left(\mathcal F^{(1)}-\mathcal F^{(2)},\chi\right)_\rho
\right|
\lesssim
e^{-\beta s}
\left(
|\Delta\ell(s)|+|\Delta x(s)|
\right)
\|\chi\|_{L^2_\rho}
\end{equation}
for every \(\chi\in L^2_\rho\). Indeed,
\[
\begin{aligned}
\mathcal F^{(1)}-\mathcal F^{(2)}
&=
\left[
\left(\lambda^{(1)}\right)^\beta
-
\left(\lambda^{(2)}\right)^\beta
\right]
f\left(\lambda^{(1)}y+x^{(1)}\right)
\\
&\quad
+
\left(\lambda^{(2)}\right)^\beta
\left[
f\left(\lambda^{(1)}y+x^{(1)}\right)
-
f\left(\lambda^{(2)}y+x^{(2)}\right)
\right].
\end{aligned}
\]
Thus,
\[
\left|
\mathcal F^{(1)}-\mathcal F^{(2)}
\right|
\lesssim
e^{-\beta s}|\Delta\ell|
+
e^{-\beta s}
\left(
e^{-s}|\Delta\ell|(1+|y|)
+
|\Delta x|
\right),
\]
which implies \eqref{eq:source-diff-Lip}, using the exponential weight
\(\rho\).

\medskip

\noindent
\emph{Step 2. Difference of the modulation equations.}
Taking the \(L^2_\rho\)-scalar product of
\eqref{eq:diff-eps-source} with
\[
\psi_1=\frac{\Lambda\Phi_n}{\|\Lambda\Phi_n\|_{L^2_\rho}},
\qquad
\psi_j,\quad 2\le j\le n+1,
\qquad
\partial_k\Phi_n,\quad 1\le k\le3,
\]
and using the orthogonality conditions \eqref{orthogonality}, combining \eqref{eq:source-diff-Lip}, we obtain
\begin{equation}\label{eq:diff-modulation-Lip}
\begin{aligned}
&
\left|
\frac{\lambda_s^{(1)}}{\lambda^{(1)}}
-
\frac{\lambda_s^{(2)}}{\lambda^{(2)}}
\right|
+
\left|
\frac{x_s^{(1)}}{\lambda^{(1)}}
-
\frac{x_s^{(2)}}{\lambda^{(2)}}
\right|
+
\sum_{j=2}^{n+1}
\left|
(\Delta a_j)_s-\mu_j\Delta a_j
\right|
\\
&\qquad
\lesssim
e^{-\mu s}
\left(
\|\Delta\varepsilon\|_{L^2_\rho}
+
\sum_{j=2}^{n+1}|\Delta a_j|
\right)
+
e^{-\beta s}
\left(
|\Delta\ell|+|\Delta x|
\right).
\end{aligned}
\end{equation}
Since
\[
\Delta\ell(s_0)=0,
\qquad
\Delta x(s_0)=0,
\]
for \(s_0\) sufficiently large, integrating \eqref{eq:diff-modulation-Lip} gives
\begin{equation}\label{eq:Delta-lambda-x}
|\Delta\ell(s)|+|\Delta x(s)|
\lesssim
\int_{s_0}^s
e^{-\mu\tau}
\left(
\|\Delta\varepsilon(\tau)\|_{L^2_\rho}
+
\sum_{j=2}^{n+1}|\Delta a_j(\tau)|
\right)
d\tau.
\end{equation}

\medskip

\noindent
\emph{Step 3. Energy estimate for the difference.}
Taking the scalar product of \eqref{eq:diff-eps-source} with
\(\Delta\varepsilon\), we get
\[
\frac12\frac{d}{ds}
\|\Delta\varepsilon\|_{L^2_\rho}^2
=
-(L_n\Delta\varepsilon,\Delta\varepsilon)_\rho
+
(\Delta\widetilde L,\Delta\varepsilon)_\rho
+
(\Delta \text{NL},\Delta\varepsilon)_\rho
+
(\Delta \text{Mod},\Delta\varepsilon)_\rho.
\]
By the spectral gap and the orthogonality conditions,
\[
-(L_n\Delta\varepsilon,\Delta\varepsilon)_\rho
\le
-c_n\|\Delta\varepsilon\|_{H^1_\rho}^2.
\]
Using \eqref{eq:diff-modulation-Lip}, the bootstrap bounds, and the source
estimate \eqref{eq:source-diff-Lip}, we obtain
\[
\begin{aligned}
&
|(\Delta\widetilde L,\Delta\varepsilon)_\rho|
+
|(\Delta \text{Mod},\Delta\varepsilon)_\rho|
\\
&\qquad
\lesssim
e^{-\mu s}
\left(
\|\Delta\varepsilon\|_{H^1_\rho}^2
+
\sum_{j=2}^{n+1}|\Delta a_j|^2
\right)
+
e^{-2\beta s}
\left(
|\Delta\ell|^2+|\Delta x|^2
\right).
\end{aligned}
\]
Moreover, the Taylor expansion of the nonlinearity gives
\[
\begin{aligned}
&
\big|
|\Phi_n+v^{(1)}|^{p-1}(\Phi_n+v^{(1)})
-
|\Phi_n+v^{(2)}|^{p-1}(\Phi_n+v^{(2)})
\\
&\qquad
-
p\Phi_n^{p-1}\Delta v
\big|
\lesssim
\left(|v^{(1)}|+|v^{(2)}|\right)|\Delta v|.
\end{aligned}
\]
Since $\|v^{(i)}(s)\|_{L^\infty}\lesssim e^{-\mu s},$
we get
\[
|(\Delta NL,\Delta\varepsilon)_\rho|
\lesssim
e^{-\mu s}
\left(
\|\Delta\varepsilon\|_{H^1_\rho}^2
+
\sum_{j=2}^{n+1}|\Delta a_j|^2
\right)
+
e^{-2\beta s}
\left(
|\Delta\ell|^2+|\Delta x|^2
\right).
\]
Taking \(s_0\) sufficiently large and using \eqref{eq:Delta-lambda-x}, the
terms containing \(\Delta\ell\) and \(\Delta x\) are perturbative. Therefore,
\begin{equation}\label{eq:diff-energy-Lip}
\frac{d}{ds}
\|\Delta\varepsilon\|_{L^2_\rho}^2
+
c_n\|\Delta\varepsilon\|_{L^2_\rho}^2
\lesssim
e^{-\mu s}
\sum_{j=2}^{n+1}|\Delta a_j|^2.
\end{equation}

\medskip

\noindent
\emph{Step 4. Reintegration of the unstable modes.}
Define
\[
M:=
\sup_{s\ge s_0}
e^{\mu s}
\sum_{j=2}^{n+1}|\Delta a_j(s)|,
\qquad
\widetilde M:=
\sup_{s\ge s_0}
e^{2\mu s}
\|\Delta\varepsilon(s)\|_{L^2_\rho}^2.
\]
From \eqref{eq:diff-modulation-Lip}, for \(2\le j\le n+1\),
\[
(\Delta a_j)_s-\mu_j\Delta a_j
=
O\left(
e^{-\mu s}
\left(
\|\Delta\varepsilon\|_{L^2_\rho}
+
\sum_{k=2}^{n+1}|\Delta a_k|
\right)
\right)
+
O\left(
e^{-\beta s}
\left(
|\Delta\ell|+|\Delta x|
\right)
\right).
\]
Reintegrating, we obtain
\[
\begin{aligned}
\Delta a_j(s)
&=
e^{\mu_j(s-s_0)}\Delta a_j(s_0)
\\
&\quad
+
e^{\mu_j s}
\int_{s_0}^{s}
e^{-\mu_j\tau}
O\left(
e^{-\mu\tau}
\left(
\|\Delta\varepsilon(\tau)\|_{L^2_\rho}
+
\sum_{k=2}^{n+1}|\Delta a_k(\tau)|
\right)
\right)
d\tau
\\
&\quad
+
e^{\mu_j s}
\int_{s_0}^{s}
e^{-\mu_j\tau}
O\left(
e^{-\beta\tau}
\left(
|\Delta\ell(\tau)|+|\Delta x(\tau)|
\right)
\right)
d\tau.
\end{aligned}
\]
Since both solutions remain trapped in the bootstrap regime, we have $
|\Delta a_j(s)|\lesssim e^{-\mu s}$.
Hence the coefficient of the exponentially growing term \(e^{\mu_j s}\)
must vanish. Therefore,
\[
\begin{aligned}
\Delta a_j(s_0)e^{-\mu_j s_0}
&+
\int_{s_0}^{\infty}
e^{-\mu_j\tau}
O\left(
e^{-\mu\tau}
\left(
\|\Delta\varepsilon(\tau)\|_{L^2_\rho}
+
\sum_{k=2}^{n+1}|\Delta a_k(\tau)|
\right)
\right)
d\tau
\\
&+
\int_{s_0}^{\infty}
e^{-\mu_j\tau}
O\left(
e^{-\beta\tau}
\left(
|\Delta\ell(\tau)|+|\Delta x(\tau)|
\right)
\right)
d\tau
=0.
\end{aligned}
\]
Using \eqref{eq:Delta-lambda-x}, we infer
\[
|\Delta a_j(s_0)|
\lesssim
e^{-2\mu s_0}\left(M+\sqrt{\widetilde M}\right).
\]
Consequently,
\[
M\lesssim e^{-\mu s_0}\sqrt{\widetilde M}.
\]
On the other hand, reintegrating the energy estimate
\eqref{eq:diff-energy-Lip} yields
\[
\widetilde M
\lesssim
e^{c_n s_0}
\|\Delta\varepsilon(s_0)\|_{L^2_\rho}^2.
\]
Combining the last two estimates gives
\[
|\Delta a_j(s_0)|
\lesssim
\|\Delta\varepsilon(s_0)\|_{L^2_\rho}.
\]
Thus,
\[
\sum_{j=2}^{n+1}
|\Delta a_j(s_0)|^2
\lesssim
\|\Delta\varepsilon(s_0)\|_{L^2_\rho}^2.
\]
Since
\[
\|\Delta\varepsilon(s_0)\|_{L^2_\rho}
\lesssim
\|\bar u_0^{(1)}-\bar u_0^{(2)}\|_{L^\infty},
\]
we obtain the desired Lipschitz dependence of the unstable parameters.

\medskip

\noindent
\emph{Step 5. Lipschitz dependence of the blow-up time.}
The previous estimates imply
\[
\sum_{j=2}^{n+1}|\Delta a_j(s)|
+
\|\Delta\varepsilon(s)\|_{L^2_\rho}
+
|\Delta\ell(s)|
+
|\Delta x(s)|
\lesssim
e^{-\mu(s-s_0)}
\|\Delta\varepsilon(s_0)\|_{L^2_\rho}.
\]
In particular,
\[
\left|
\frac{d}{ds}
\log\frac{\lambda^{(1)}(s)}{\lambda^{(2)}(s)}
\right|
\lesssim
e^{-\mu(s-s_0)}
\|\Delta\varepsilon(s_0)\|_{L^2_\rho}.
\]
Since
\[
\lambda^{(1)}(s_0)=\lambda^{(2)}(s_0)=\lambda_0,
\]
integration gives
\[
\left|
\log\frac{\lambda^{(1)}(s)}{\lambda^{(2)}(s)}
\right|
\lesssim
\|\Delta\varepsilon(s_0)\|_{L^2_\rho}.
\]
Thus,
\[
\lambda^{(2)}(s)
=
\lambda^{(1)}(s)
\left(
1+
O\left(
\|\Delta\varepsilon(s_0)\|_{L^2_\rho}
\right)
\right).
\]
Since
\[
T^{(i)}
=
\int_{s_0}^{\infty}
\left(\lambda^{(i)}(s)\right)^2\,ds,
\]
we obtain
\[
\begin{aligned}
|T^{(1)}-T^{(2)}|
&\le
\int_{s_0}^{\infty}
\left|
\left(\lambda^{(1)}(s)\right)^2
-
\left(\lambda^{(2)}(s)\right)^2
\right|\,ds
\\
&\lesssim
\|\Delta\varepsilon(s_0)\|_{L^2_\rho}
\int_{s_0}^{\infty}
\left(\lambda^{(1)}(s)\right)^2\,ds
\\
&\lesssim
\|\bar u_0^{(1)}-\bar u_0^{(2)}\|_{L^\infty}.
\end{aligned}
\]
This proves the Lipschitz dependence of the blow-up time and completes the
proof.
\end{proof}

\section*{Acknowledgements}
The author is grateful to Prof. Charles Collot for helpful discussions.  K. Zhang is supported by China Scholarship Council (No.202206460045).

\section*{Data Availability Statement}
Data sharing is not applicable to this article as no datasets were generated or analysed during the current study.

\end{document}